\documentclass[11pt]{article}

\usepackage[T2A]{fontenc}
\usepackage[cp1251]{inputenc}
\usepackage{amsthm}
\usepackage{amsfonts}
\usepackage{amssymb}
\usepackage{amsmath}
\usepackage{cite}

\makeatletter
\renewcommand{\@biblabel}[1]{#1.\hfill}
\makeatother

\makeatother
\usepackage{bm}
\begin{document}
\newtheoremstyle{mytheorem}
  {\topsep}   
  {\topsep}   
  {\itshape}  
  {}       
  {\bfseries} 
  {  }         
  {5pt plus 1pt minus 1pt} 
  { }          
\newtheoremstyle{myremark}
  {\topsep}   
  {\topsep}   
  {\upshape}  
  {}       
  {\bfseries} 
  {  }         
  {5pt plus 0pt minus 1pt} 
  {}          
\theoremstyle{mytheorem}
\newtheorem*{A}{Heyde theorem}
\newtheorem{theorem}{Theorem}[section]
 \newtheorem{theorema}{Theorem}
 \renewcommand{\thetheorema}{\Alph{theorema}}
 \newtheorem{proposition}[theorem]{Proposition}
 \newtheorem{lemma}[theorem]{Lemma}
\newtheorem{corollary}[theorem]{Corollary}
\newtheorem{definition}[theorem]{Definition}
\theoremstyle{myremark}
\newtheorem{remark}[theorem]{Remark}
\noindent{\Large\textbf{An Analogue of Heyde's Theorem for Discrete Torsion}}

\medskip

\noindent{\Large\textbf{Abelian Groups with Cyclic $p$-Components}}

\bigskip

\noindent\textbf{Gennadiy Feldman} (ORCID ID  https://orcid.org/0000-0001-5163-4079)

\bigskip

\noindent\textbf{Abstract} 

\noindent According to the well-known Heyde theorem, the  
Gaussian distribution  on the real line is characterized by 
the symmetry of the conditional  distribution of one linear form of 
$n$ independent random variables given another. In the article, we prove an 
analogue  of this theorem for two independent random variables  taking values
in a discrete torsion Abelian group $X$ with cyclic $p$-components.
 In doing so, we do not 
impose any restrictions on  coefficients of the linear forms 
 and the characteristic functions of random variables. The proof 
 uses methods of abstract harmonic analysis and is based on the solution 
 some functional equation on the character group of the group $X$.

\bigskip

  
  \bigskip
  

\bigskip

\section { Introduction}

Characterization theorems in mathematical statistics are statements in which the 
description of possible distributions of random variables follows from properties 
of some functions of these variables. One of the most famous characterization 
theorems was proved independently by M.~Kac and S.N.~Bernstein: if the sum and difference of two independent random variables are independent, 
then the random variables are Gaussian. Many characterization theorems have been studied 
in the case when random variables took values in a locally compact Abelian group.
  Among theorems whose 
group analogues 
are being actively studied   
  is the well-known Heyde theorem \cite{He}, 
see also \cite[Theorem 13.4.1]{KaLiRa}.
 For two independent random variables this theorem  can be 
 formulated as follows: 
 \begin{A}{Let
 $\xi_1$ and $\xi_2$ be real-valued independent random
 variables with distributions $\mu_j$. Let $\alpha_j$, $\beta_j$ be 
 nonzero constants such that
 $\beta_1\alpha_1^{-1}+\beta_2\alpha_2^{-1} \ne 0$.   
 If the conditional distribution of the linear form $L_2 = \beta_1\xi_1 +
   \beta_2\xi_2$   given $L_1 = \alpha_1\xi_1 +  \alpha_2\xi_2$ is 
   symmetric\footnote{We note that the conditional  distribution of  $L_2$ 
   given $L_1$ is symmetric
if and only if the random vectors $(L_1, L_2)$ and $(L_1, -L_2)$ are 
identically distributed.}, 
   then   $\mu_j$ are Gaussian distributions.}
 \end{A}

A group analogue of Heyde's theorem was first proved in 
\cite{F2004} for finite Abelian groups. Other classes of locally compact Abelian groups 
were then studied as well, in particular discrete groups, compact totally disconnected groups, 
${\bm a}$-adic solenoids\footnote{On group analogues of Heyde's theorem see, 
e.g., \cite{ {F_Heyde_new}, 
 {FeTVP}, {Rima}, {JFAA2024},
 {M2023}, {M2025}}, and also \cite{Febooknew}, where one can find additional references.}. 
In doing so, it is usually assumed that the 
coefficients of linear forms are topological automorphisms of the group. It is 
easy to see that in this case we can suppose, without loss of generality, 
 that $L_1 = \xi_1 + \xi_2$ and 
$L_2 = \xi_1 + \alpha\xi_2$, where $\alpha$ is a topological 
automorphism of the group. It turns out that for a given locally compact Abelian 
group, the description of distributions 
that are characterized by the 
symmetry of the conditional distribution of one linear form given another 
depends on whether the kernel $\mathrm{Ker}(I+\alpha)$
 is zero and on whether the characteristic functions of the distributions 
$\mu_j$ vanish.

In article \cite{F_Heyde_new}, for the first time a group analogue of 
Heyde's theorem was studied without any restrictions on $\alpha$ and the characteristic 
functions of the distributions $\mu_j$.
In \cite{F_Heyde_new}
independent random variables  take values either in a 
compact totally disconnected Abelian group of a certain class 
or in a $p$-quasicyclic group.

In the present article we prove an analogue  of Heyde's theorem for two
independent random variables taking values either in a discrete torsion Abelian 
group  with cyclic $p$-components or in the direct product of the group of real 
numbers and a discrete torsion Abelian group  with cyclic $p$-components.  
As in article \cite{F_Heyde_new}, we do not impose any restrictions    
on   $\alpha$ and the characteristic functions of the distributions $\mu_j$.
It should be noted that the original Heyde's theorem was proven for an arbitrary 
number $n$ of independent random variables, subject to certain restrictions on the 
coefficients of the linear forms. However, as proven in 
\cite{F2004}, even for finite Abelian groups, roughly speaking, 
when $n>2$, a reasonable analogue of Heyde's theorem does not exist.

We would like to emphasize that despite the probabilistic-statistical formulation, 
the problems under 
consideration are problems of abstract harmonic analysis. Description of possible 
distributions $\mu_j$ in a group analogue of Heyde's theorem for a locally 
compact Abelian group $X$ is equivalent to description of solutions of so called
Heyde's functional 
equation in the class of characteristic functions (Fourier transform) 
of probability 
distributions on the character group of the group $X$. It should be noted that the Heyde 
equation is in the same series  with such well-known functional equations 
as the Kac-Bernstein and the Skitovich-Darmois equations, which are closely 
related to the corresponding characterization theorems (see 
\cite[Chapters II and III]{Febooknew}). These and similar equations 
have been studied by many authors independently of characterization problems
(see, e.g., J.M.~Almira \cite{A1}, \cite{A2}, 
J.M.~Almira and E.V.~Shulman \cite{AS1}, 
M.~Sablik and E.~Shulman \cite{SaSh1}, E.V.~Shulman \cite{Sh1}).

In the article we use standard results of abstract harmonic analysis 
 (see, e.g., \cite{Hewitt-Ross}). Let $X$ be a locally compact Abelian group and $Y$ be 
 its the character
group. We also use the notation $X^*$ for the character group of $X$. 
Let  $x \in X$ and $y \in Y$.  Denote by  $(x,y)$ the value of 
the character $y$ at the element  $x$. For a  subgroup $K$ of the group $X$,  denote by
$A(Y, K) = \{y \in Y: (x, y) =1$ {for all} $x\in
K\}$
its annihilator. 
 Denote by $\mathrm{Aut}(X)$ the group
of all topological automorphisms of $X$  and by  $I$ the identity automorphism.   
Let $G$ be a closed subgroup of the group $X$ and let $\alpha\in\mathrm{Aut}(X)$.
If $\alpha(G)=G$, i.e., the restriction  of  $\alpha$ to $G$ is a topological 
automorphism of the group   $G$, then we denote by   $\alpha_{G}$ this restriction.
A closed subgroup $G$ of $X$ is called characteristic if $\alpha(G)=G$ for 
all $\alpha\in\mathrm{Aut}(X)$. 
For $\alpha\in\mathrm{Aut}(X)$, the adjoint automorphism 
$\widetilde\alpha\in \mathrm{Aut}(Y)$
is defined by the formula $(\alpha x,
y)=(x, \widetilde\alpha y)$ for all $x\in X$, $y\in
Y$.  For a natural $n$, 
denote by $f_n$ 
the  endomorphism of the group $X$
 defined by the formula:  $f_nx=nx$, $x\in X$. Put $X^{(n)}=f_n(X)$.
 Let $x\in X$ be an element  of finite 
order. Denote by $\langle x\rangle$ the subgroup of $X$ generated by $x$.
Denote by $\mathbb{R}$ the group of real numbers and by
 $\mathbb{Z}(n)$ the  group of the integers 
modulo $n$, i.e., the finite cyclic group of order $n$.

Let $Y$ be an Abelian group, let $f(y)$ be a function on  $Y$, and let $h$ be an
element of $Y$. Denote by $\Delta_h$ the finite difference operator 
$$\Delta_h f(y)=f(y+h)-f(y), \quad y\in Y.$$ A function $f(y)$ on $Y$ is called
a  polynomial  if
$$\Delta_{h}^{n+1}f(y)=0$$ for some nonnegative integer 
$n$ and all $y, h \in Y$.

Let $X$ be a locally compact Abelian group and let $\mu$ and $\nu$ be probability 
distributions on $X$. The convolution
$\mu*\nu$ is defined by the formula
$$
\mu*\nu(B)=\int\limits_{X}\mu(B-x)d \nu(x)
$$
for any Borel subset $B$ of $X$. 
Denote by
$$
\hat\mu(y) =
\int\limits_{X}(x, y)d \mu(x), \quad y\in Y,$$   
the characteristic function 
(Fourier transform) of 
the distribution $\mu$.
 Define the distribution 
 $\bar \mu $ by the formula
 $\bar \mu(B) = \mu(-B)$ for any Borel  subset $B$ of $X$.
Then $\hat{\bar{\mu}}(y)=\overline{\hat\mu(y)}$. A distribution   
$\mu_1$ on the group $X$ is called a {factor} of $\mu$  if there is a distribution   
$\mu_2$ on $X$ such that $\mu=\mu_1*\mu_2$.
For $x\in X$, denote by $E_x$  the degenerate distribution
 concentrated at the element $x$. 
We say that a function $\psi(y)$ on the group
$Y$ is characteristic if there is a distribution $\mu$ on the group $X$ such that
$\psi(y)=\hat\mu(y)$ for all $y\in Y$. 
For a compact subgroup $K$ of the group $X$ denote by $m_K$ the Haar distribution on $K$. 
The
characteristic function  $\widehat m_K(y)$ is of the form
\begin{equation}
\label{fe22.1}\widehat m_K(y)=
\begin{cases}
1 & \text{\ if\ }\ \ y\in A(Y,K),
\\ 0 & \text{\ if\ }\ \ y\not\in
A(Y,K).
\end{cases}
\end{equation}

\section{ Main theorem}

Let $\{H_\iota: \iota\in {\mathcal I}\}$ be a family of
discrete Abelian groups. Denote by
$\mathop{\mathbf{P}^*}\limits_{\iota \in {\mathcal I}}H_\iota$ 
the {weak direct product of the
groups} $H_\iota$, considering in the discrete topology. 
Let $\{G_\iota: \iota\in {\mathcal I}\}$ be a family of
compact Abelian groups. Denote by $\mathop\mathbf{P}\limits_{\iota \in
	{\mathcal I}}G_\iota$  the {direct product of the groups}
$G_\iota$ considering in the product topology. 
Let $X$ be an Abelian group and let $p$ be a prime number.
A group  $X$ is called  $p$-{group}  if the order of every 
element of $X$ is a power of $p$. If
 $X$ is a torsion group,
the subgroup of $X$ consisting of all elements of $X$ whose order is a power of 
$p$ is called the $p$-component of $X$. 

The main result of the article is the following theorem.

\begin{theorem}\label{tn1} Let $X$ be a discrete torsion Abelian group 
with cyclic $p$-components containing no elements of order $2$.
 Let  $\alpha$ be an automorphism of the group $X$. 
 Let $\xi_1$ and $\xi_2$ be independent random variables with values in $X$ 
 and distributions $\mu_1$ and $\mu_2$. Assume that the conditional  distribution 
 of the linear form $L_2 = \xi_1 + \alpha\xi_2$ given $L_1 = \xi_1 + \xi_2$ is 
symmetric. Then  there is a subgroup $G$ of the group $X$ and  
a distribution $\lambda$ supported in $G$ such that 
the following statements are true:
\renewcommand{\labelenumi}{\rm(\roman{enumi})}
\begin{enumerate}
\item
$\mu_j$ are shifts of $\lambda$;
	\item
$G$ is the minimal subgroup containing the support of $\lambda$;
\item
$(I+\alpha)(G)$ is a finite group;
\item
the Haar distribution $m_{(I+\alpha)(G)}$ is a factor of  $\lambda$;
\item
if $\eta_j$ are independent identically distributed random variables 
with values in    
the group $X$  and distribution  $\lambda$, then    the conditional  
distribution of the linear form
$M_2=\eta_1 + \alpha\eta_2$ given $M_1=\eta_1 + \eta_2$ is symmetric.
\end{enumerate} 
\end{theorem}

For the proof of the theorem we need some lemmas.
It is convenient for us to formulate   the following   well-known 
statement  in the form as a lemma 
(for the proof see, e.g., \cite[Proposition 2.10]{Febooknew}).
\begin{lemma}\label{lem2}  Let $X$ be a locally compact Abelian group  
with character group
 $Y$
and let $\mu$ be a distribution on $X$. Then the sets 
$$E=\{y\in Y:  \hat\mu(y)=1\}, \quad
B=\{y\in Y:  |\hat\mu(y)|=1\}$$ are closed subgroups of the group $Y$, and 
the distribution   $\mu$ is supported in the subgroup $A(X,E)$.
\end{lemma}

\begin{lemma}[{\!\!\protect\cite{Rima}, see also 
		\cite[Lemma 9.10]{Febooknew}}]
\label{lem11} 
Let $X$ be a second countable locally compact Abelian group with character group
 $Y$. Let $H$ be a closed subgroup of $Y$, and let $\alpha$ be a topological 
 automorphism 
 of $X$. Put $G=A(X, H)$. Assume that $H^{(2)}=H$ and  $\alpha(G)=G$.
Let $\xi_1$ and  $\xi_2$  be independent random variables with values in
the group $X$  and distributions $\mu_1$ and $\mu_2$ such that
$$|\hat\mu_1(y)|=|\hat\mu_2(y)|=1, \quad y\in H.$$  
Suppose that the conditional  distribution of the linear form 
$L_2 = \xi_1 + \alpha\xi_2$ given $L_1 = \xi_1 + \xi_2$ is symmetric. 
Then there are some shifts $\lambda_j$ of the distributions $\mu_j$  
such that $\lambda_j$ are supported in $G$. In doing so, if  $\eta_j$ 
are independent random variables with values in   $X$  and distributions 
$\lambda_j$, then the conditional  distribution of the linear form
$M_2=\eta_1 + \alpha\eta_2$ given $M_1=\eta_1 + \eta_2$ is symmetric.
\end{lemma}

\begin{lemma}[{\!\!\protect\cite[Lemma 9.1]{Febooknew}}]
\label{lem1}
 Let $X$ be a second countable locally compact Abelian group with character group
 $Y$ and let $\alpha$ be a topological automorphism of $X$.
Let
$\xi_1$ and  $\xi_2$  be independent random variables with values in
 the group $X$  and distributions $\mu_1$ and $\mu_2$.   
 Then the conditional distribution of the linear form 
 $L_2 = \xi_1 + \alpha\xi_2$ given $L_1 = \xi_1 + \xi_2$ is symmetric if and only
 if the characteristic functions
 $\hat\mu_j(y)$ satisfy the equation\footnote{Equation (\ref{11.04.1}) 
 is a particular case 
for two independent random variables of so called  the Heyde functional equation.}
\begin{equation}\label{11.04.1}
\hat\mu_1(u+v )\hat\mu_2(u+\widetilde\alpha v )=
\hat\mu_1(u-v )\hat\mu_2(u-\widetilde\alpha v), \quad u, v \in Y.
\end{equation}
\end{lemma}

The following lemma is a particular case for 
finite cyclic groups of odd order of Corollary 2.6 in \cite{F_Heyde_new}.
This lemma plays a key role in proving Theorem \ref{tn1}.

\begin{lemma}\label{nth1}  
Let $X$ be a finite cyclic group 
of odd order with character group $Y$. Let  $\alpha$ be an automorphism of the group $X$. 
 Let $\xi_1$ and $\xi_2$ be
independent random variables with values in   $X$ and distributions
$\mu_1$ and $\mu_2$ such that 
\begin{equation}\label{ne20.35}
\{y\in Y:|\hat\mu_1(y)|=|\hat\mu_2(y)|=1\}=\{0\}.
\end{equation}
Assume that the conditional  distribution of 
the linear form $L_2 = \xi_1 + \alpha\xi_2$ given $L_1 = \xi_1 + \xi_2$ 
is symmetric.
Then  $\mu_1=\mu_2=\mu$  and the 
Haar distribution $m_{(I+\alpha)(X)}$ 
is a factor of $\mu$.
\end{lemma}

The following statement is a particular case of Lemma 9.17 in \cite{Febooknew}.

\begin{lemma}
\label{newle18.01.2}  
 Let $Y$
be an Abelian group  and let $\beta$ be an automorphism of $Y$.
Assume that the function  $\varphi(y)$ satisfies the equation 
\begin{equation}\label{e20.30}
\varphi(u+v)+\varphi(u+\beta v)=\varphi(u-v)+
\varphi(u-\beta v)=0,\quad u, v \in
 Y.
\end{equation}
 Then  the function  $\varphi(y)$ satisfies the equation 
 \begin{equation*}\label{e20.31}
\Delta_{(I-\beta)k_3}\Delta_{2 k_2}\Delta_{(I+\beta)k_1}\varphi(y)
= 0,\quad  y\in Y,
\end{equation*}
 where $k_j$, $j=1, 2, 3$, are arbitrary elements of the group $Y$. 
\end{lemma}

The following lemma is well known (for the proof see, 
e.g., \cite[Proposition 1.30]{Febooknew}).
\begin{lemma}
\label{s5.6}    
	Let $Y$
	be a  compact Abelian group   and let $f(y)$ be a continuous
	polynomial on $Y$. Then $f(y)=\text{const}$
		for all $y\in Y$.
\end{lemma}

\noindent\textbf{\emph{Proof of Theorem  \ref{tn1}}}  
Any torsion Abelian group is isomorphic to a weak direct product of its 
$p$-components (\!\!\cite[(A.3)]{Hewitt-Ross}).
By the condition of the theorem, each $p$-component of the group $X$ 
is cyclic. Taking into 
account that any cyclic $p$-group is isomorphic to $\mathbb{Z}(p^{k})$
for some natural $k$,
we can suppose, without loss of generality, that
\begin{equation}\label{12.08.3}
X=\mathop{\mathbf{P}^*}\limits_{p_j \in {\mathcal P}}\mathbb{Z}(p_j^{k_j}),
\end{equation}
where ${\mathcal P}$ is a set of pairwise distinct prime numbers such that 
$2\notin {\mathcal P}$ and all $k_j\ge 1$.

Denote by $Y$ the character group of the group $X$. The group $Y$ is compact and 
topologically 
isomorphic to the direct product of the groups 
$\mathbb{Z}(p_j^{k_j})$, where ${p_j \in {\mathcal P}}$. 
To avoid introducing 
additional notation, we assume that
\begin{equation}\label{12.08.1} 
Y=\mathop{\mathbf{P}}\limits_{p_j \in {\mathcal P}}\mathbb{Z}(p_j^{k_j}).
\end{equation}
It is easy to see that any subgroup $K$ of the group $X$ is of the form
\begin{equation}\label{nen1}
K=\mathop{ \mathbf{P}^*}\limits_{p_j \in {\mathcal S}}\mathbb{Z}(p_j^{l_j}),
\end{equation}
where $\mathcal S\subset \mathcal P$  and $l_j\le k_j$, i.e., 
$K$ is also a discrete torsion Abelian group 
with cyclic $p$-components containing no elements of order $2$.
Denote by $x=(x_1, x_2,\dots, x_n, 0, 0,\dots)$, 
where $x_j\in\mathbb{Z}(p_j^{k_j})$, elements of the group $X$.
Let $\alpha\in\mathrm{Aut}(X)$. Since $p_i\ne p_j$  for 
all $i\ne j$, the automorphism $\alpha$ acts on elements of the group 
$X$ as follows:
\begin{equation}\label{en21}
\alpha(x_1, x_2,\dots, x_n, 0, 0,\dots)=
(\alpha_{\mathbb{Z}(p_1^{k_1})}x_1, \alpha_{\mathbb{Z}(p_2^{k_2})}x_2,\dots, 
\alpha_{\mathbb{Z}(p_n^{k_n})}x_n, 0, 0,\dots).
\end{equation} 
Note  that each automorphism 
of the group $\mathbb{Z}(p_j^{k_j})$ is the multiplication by 
a  natural number $m$, i.e., coincides with an endomorphism $f_{m}$, 
where $m$ and $p_j$ are mutually prime. 
In view of (\ref{nen1}) and (\ref{en21}), any subgroup of the group $X$ is characteristic. 
From the above it follows that we can assume, without loss of 
generality, that the minimal subgroup of the group $X$ containing the supports 
of $\mu_1$  and  $\mu_2$ coincides with $X$. 

Consider the set
$$H=\{y\in Y:|\hat\mu_1(y)|=|\hat\mu_2(y)|=1\}.$$
By Lemma \ref{lem2}, $H$ is a closed subgroup of the group $Y$.
Since $f_2$ is a topological automorphism of any closed subgroup of
the group $Y$, we have $H^{(2)}=H$.
Taking into account that any subgroup of the group $X$ is characteristic,
we can apply Lemma \ref{lem11} and reduce the proof of the theorem
to the case when condition (\ref{ne20.35}) 
is fulfilled, i.e., $H=\{0\}$.
We will prove that in this case there is a distribution $\mu$ such that
the following statement are true:
\renewcommand{\labelenumi}{\rm(\Roman{enumi})}
\begin{enumerate}
\item
$\mu_1=\mu_2=\mu$;
	\item
$X$ is the minimal subgroup containing the support of $\mu$;
\item
$(I+\alpha)(X)$ is a finite
group;
			\item
			the Haar distribution $m_{(I+\alpha)(X)}$ is a factor of  $\mu$.
\end{enumerate}
Then the theorem will be proved. Thus, assuming that condition (\ref{ne20.35})
holds, we will prove statements (I)--(IV).

Let us prove statement (I). By Lemma \ref{lem1}, 
the characteristic functions $\hat\mu_j(y)$   satisfy equation (\ref{11.04.1}).
The restriction of equation (\ref{11.04.1}) 
to the subgroup   $\mathrm{Ker}(I+\widetilde\alpha)$  is of the form 
\begin{equation}\label{12.08.5}
\hat\mu_1(u+v )\hat\mu_2(u-v )=
\hat\mu_1(u-v )\hat\mu_2(u+v), \quad u, v \in \mathrm{Ker}(I+\widetilde\alpha).
\end{equation}
Substituting $u=v=y$ in equation (\ref{12.08.5}), we obtain 
$\hat\mu_1(2y)=\hat\mu_2(2y)$ for all $y \in \mathrm{Ker}(I+\widetilde\alpha)$.
Since $f_2$ is a topological automorphism of any closed subgroup of
the group $Y$, this implies that
$\hat\mu_1(y)=\hat\mu_2(y)$ for all $y \in \mathrm{Ker}(I+\widetilde\alpha)$. Set
\begin{equation}\label{en5}
\psi(y)=\hat\mu_1(y)=\hat\mu_2(y), \quad y \in \mathrm{Ker}(I+\widetilde\alpha).
\end{equation}

Let us verify that 
\begin{equation}\label{e1}
\hat\mu_j(y)=0 \ \mbox{for all} \ y \notin \mathrm{Ker}(I+\widetilde\alpha), \ j=1, 2.
\end{equation}
  Take $y_0\notin \mathrm{Ker}(I+\widetilde\alpha)$. 
Assume first that $y_0$ is an element of finite order.
Consider the subgroup
$T=\langle  y_0 \rangle$  generated by the 
element $y_0$. By the Pontryagin duality theorem, $T$ is the 
character group of a finite cyclic group $S$, where
 $S$ is isomorphic to $T$. 
Denote by  $\omega_j$ the  distributions on the group $S$ 
with the characteristic functions
\begin{equation}\label{22en6}
\hat\omega_j(y)=\hat\mu_j(y), \quad y\in T, \ j=1, 2.
\end{equation}
Since the characteristic functions $\hat\mu_j(y)$ 
 satisfy equation (\ref{11.04.1}) and the subgroup $T$  is characteristic, 
 the characteristic functions 
 $\hat\omega_j(y)$ satisfy the equation
\begin{equation}\label{14.08.3}
\hat\omega_1(u+v )\hat\omega_2(u+\widetilde\kappa v )=
\hat\omega_1(u-v )\hat\omega_2(u-\widetilde\kappa v), \quad u, v \in T,
\end{equation} 
where $\kappa$ is an automorphism of the group $S$. 
It follows from (\ref{ne20.35})  and (\ref{22en6}) that
\begin{equation}\label{07.08.5}
\{y\in T:|\hat\omega_1(y)|=|\hat\omega_2(y)|=1\}=\{0\}.
\end{equation}

Let $\zeta_1$ and $\zeta_2$ be independent random variables with values in the group
$S$ and distributions $\omega_1$ and $\omega_2$. 
By  Lemma \ref{lem1}, it follows from (\ref{14.08.3}) that 
the conditional  distribution of 
the linear form $N_2 = \zeta_1 + \kappa\zeta_2$ given 
$N_1 = \zeta_1 + \zeta_2$ 
is symmetric. Since (\ref{07.08.5}) holds, 
we can apply Lemma 
\ref{nth1} to the group $S$, the automorphism $\kappa$,   
the random variables $\zeta_j$, and the distributions $\omega_j$.
By Lemma 
\ref{nth1}, we get that there is a distribution
$\omega$ on the group $S$ such that $\omega_1=\omega_2=\omega$ and 
the Haar distribution $m_{(I+\kappa)(S)}$ 
is a factor of of $\omega$. 
In view of (\ref{fe22.1}) and the fact that
\begin{equation*}\label{06.08.3}
A(T, (I+\kappa)(S))=\mathrm{Ker}(I+\widetilde\kappa),
\end{equation*}
from the above it follows that
$\hat\omega(y)=0$ for all $y\notin \mathrm{Ker}(I+\widetilde\kappa)$. 
We   have $\widetilde\alpha y=\widetilde\kappa y$ for all $y\in T$.
This implies that if $y_0\notin \mathrm{Ker}(I+\widetilde\alpha)$, then
$y_0\notin \mathrm{Ker}(I+\widetilde\kappa)$. Hence
 $\hat\omega(y_0)=0$ 
and $\omega_j(y_0)=\omega(y_0)=0$. 
Considering (\ref{22en6}), we get that
$\hat\mu_j(y_0)=0$, $j=1, 2$. Thus, we proved (\ref{e1}) if
$y$ is an element of finite order.

Denote by 
$y=(y_1, y_2,\dots, y_n,\dots)$, 
where $y_j\in\mathbb{Z}(p_j^{k_j})$, elements of the group $Y$. 
Assume now that $y_0=(y_1, y_2, \dots, y_n, \dots)$ is an element 
of infinite order. Put $y_0^{(n)}=(y_1, y_2, \dots, y_n, 0, 0, \dots)$.
The elements $y_0^{(n)}$ are of finite order and
$y_0^{(n)}\rightarrow y_0$ as $n\rightarrow\infty$.
It follows from $y_0 \notin \mathrm{Ker}(I+\widetilde\alpha)$ that
$y_0^{(n)}\notin \mathrm{Ker}(I+\widetilde\alpha)$ for all large enough
$n$. As proven above
$\hat\mu_1(y_0^{(n)})=\hat\mu_2(y_0^{(n)})=0$.
Hence 
$\hat\mu_1(y_0)=\hat\mu_2(y_0)=0$. Thus, (\ref{e1}) is proved.
In view of (\ref{en5}),  as a result we obtain
$\hat\mu_1(y)=\hat\mu_2(y)$ for all $y\in Y$.
Hence $\mu_1=\mu_2$. Put
\begin{equation}\label{nf1}
\mu=\mu_1=\mu_2.
\end{equation}
Thus, statement  (I) is proved.

In view of (\ref{nf1}), statement (II) follows from condition (\ref{ne20.35}).

Let us prove statement (III). By Lemma \ref{lem1}, 
the characteristic functions $\hat\mu_j(y)$   satisfy equation (\ref{11.04.1}). 
In view of (\ref{nf1}), write equation (\ref{11.04.1}) in the form
\begin{equation}\label{n11.04.1}
\hat\mu(u+v )\hat\mu(u+\widetilde\alpha v )=
\hat\mu(u-v )\hat\mu(u-\widetilde\alpha v), \quad u, v \in Y.
\end{equation}
Moreover,  condition (\ref{ne20.35}) takes the form
\begin{equation}\label{27ne20.35}
\{y\in Y:|\hat\mu(y)|=1\}=\{0\}.
\end{equation}
Put $P=\mathrm{Ker}(I-\widetilde\alpha)$.    
It follows from $\widetilde\alpha y=y$ for all $y\in P$ that  
the restriction of equation (\ref{n11.04.1}) to the subgroup $P$ is of the form
\begin{equation}\label{11.04.2}
\hat\mu^2(u+v)=\hat\mu^2(u-v), \quad u, v\in P.
\end{equation}
Substituting  $u=v=y$ in (\ref{11.04.2}), we get 
\begin{equation}\label{n11.04.2}
\hat\mu^2(2y)=1, \quad y\in P.
\end{equation}
Since $f_2$ is a topological automorphism of any closed subgroup of
the group $Y$, (\ref{n11.04.2})  implies that 
\begin{equation}\label{ynn11.04.2}
|\hat\mu(y)|=1, \quad y\in P.
\end{equation}
Taking into account (\ref{27ne20.35}) and (\ref{ynn11.04.2}), 
we conclude that $P=\mathrm{Ker}(I-\widetilde\alpha)=\{0\}$. 
Inasmuch as the group $Y$ is of the form (\ref{12.08.1}), 
this implies that 
\begin{equation}\label{14.08.1}
I-\widetilde\alpha\in\mathrm{Aut}(Y).
\end{equation}
Note that then $I-\alpha\in\mathrm{Aut}(X)$.

Put $\nu=\mu*\bar\mu$. Then 
$\hat\nu(y)=|\hat\mu(y)|^2\ge 0$  for all $y\in Y$.
Since the characteristic function  $\hat\mu(y)$   satisfies equation (\ref{n11.04.1}), 
 the characteristic function    $\hat\nu(y)$ also satisfies equation   
(\ref{n11.04.1}). 

Let ${\mathcal P}=\{p_1, p_2, \dots, p_n, \dots\}$, where
$p_i<p_j$ for $i<j$.
Taking into account that   the family of the subgroups
\begin{equation*} 
Y_k=\mathop{\mathbf{P}}\limits_{p_j \in {\mathcal P}, \ p_j\ge k}\mathbb{Z}(p_j^{k_j}),
\quad k=1,2, \dots,
\end{equation*}
forms an open basis at the zero of the group $Y$, we can choose a natural 
 $l$ in such a way that 
$\hat\nu(y)>0$ for all
$y\in Y_l$. Put $\varphi(y)=\ln\hat\nu(y)$, 
$y\in Y_l$.
Inasmuch as the characteristic function    $\hat\nu(y)$ satisfies equation   
(\ref{n11.04.1}), the function   $\varphi(y)$ satisfies equation 
(\ref{e20.30}), where   $Y=Y_l$ and
$\beta=\widetilde\alpha$. In view of 
$f_2\in\mathrm{Aut}(Y)$ and (\ref{14.08.1}), 
it follows from Lemma \ref{newle18.01.2} that 
 the function   $\varphi(y)$ satisfies the equation
$$
\Delta_{h}^3\varphi(y) = 0,
 \quad y, h\in (I+\widetilde\alpha)(Y_l),  
$$
i.e., the function  $\varphi(y)$ is a continuous 
polynomial  on the group $(I+\widetilde\alpha)(Y_l)$.
Since $\varphi(0)=0$ and the subgroup $(I+\widetilde\alpha)(Y_l)$ is compact, 
we obtain from Lemma \ref{s5.6} that
$\varphi(y)=0$ for all  
$y\in (I+\widetilde\alpha)(Y_l)$. Hence
$|\hat\mu(y)|=1$ for all $y\in (I+\widetilde\alpha)(Y_l)$. 
In view of (\ref{27ne20.35}), we have 
\begin{equation}\label{12.08.2}
(I+\widetilde\alpha)(Y_l)=\{0\}, \quad j=1, 2.
\end{equation}
Put
\begin{equation}\label{12.08.4}
X_k=\mathop{\mathbf{P}^*}\limits_{p_j \in {\mathcal P}, \ p_j\ge k}\mathbb{Z}(p_j^{k_j}),
\quad k=1,2, \dots 
\end{equation}
Inasmuch as $Y_k=X_k^*$, it follows from (\ref{12.08.2}) that 
\begin{equation}\label{12.08.7}
\alpha_{X_l}=-I.
\end{equation} 
Taking into account (\ref{12.08.3}), (\ref{12.08.4}), and (\ref{12.08.7}), we get that
 $$(I+\alpha)(X)=\mathop{\mathbf{P}}\limits_{p_j \in {\mathcal P}, 
 \ p_j<l}\mathbb{Z}(p_j^{k_j})$$ is a finite group. Thus, statement 
 (III) is proved.
 
Let us prove that statement  (IV)  is also valid. 
 In follows from (\ref{en5}), (\ref{e1}), 
and (\ref{nf1}), that the 
characteristic function of the distribution $\mu$ can be written the form
\begin{equation}
 \label{en10}\hat\mu(y)=
 \begin{cases}
 \psi(y) & \text{\ if\ }\ \ y\in\mathrm{Ker}(I+\widetilde\alpha),
 \\ 0 & \text{\ if\ }\ \ y\notin\mathrm{Ker}(I+\widetilde\alpha).
 \end{cases}
 \end{equation}
Consider the Haar distribution $m_{(I+\alpha)(X)}$.
Taking into account the fact that 
$A(Y, (I+\alpha)(X))=\mathrm{Ker}(I+\widetilde\alpha)$ and (\ref{fe22.1}), 
the characteristic 
function $\widehat m_{(I+\alpha)(X)}(y)$
is of the form
\begin{equation}\label{d2}
\widehat m_{(I+\alpha)(X)}(y)=
\begin{cases}
1 & \text{\ if\ }\ \  y\in\mathrm{Ker}(I+\widetilde\alpha),
\\ 0 & \text{\ if\ }\ \ y\notin\mathrm{Ker}(I+\widetilde\alpha).
\end{cases}
\end{equation}
It follows from
(\ref{en10}) and (\ref{d2}) that $\hat\mu(y)=
\widehat m_{(I+\alpha)(X)}(y)\hat\mu(y)$ for all $y\in Y$.
Hence $\mu=m_{(I+\alpha)(X)}*\mu$ and statement (IV) is also proved.
The theorem is completely proved. $\hfill\Box$

\medskip

The following statement results from the proof of Theorem \ref{tn1}.

\begin{corollary}\label{nnco1} Let $X$ be a discrete torsion Abelian group 
with cyclic $p$-components containing no elements of order $2$.
Let $Y$ be a character group of the group $X$. 
Assume that all conditions of 
Theorem $\ref{tn1}$ are fulfilled and the characteristic functions 
$\hat\mu_j(y)$ satisfy condition $(\ref{ne20.35})$.
Then $\mu_1=\mu_2=\mu$, $X$ is the minimal subgroup containing 
the support of 
$\mu$, $(I+\alpha)(X)$ is a finite
group, the Haar distribution $m_{(I+\alpha)(X)}$ 
is a factor  of $\mu$, and $I-\alpha\in\mathrm{Aut}(X)$.
\end{corollary}

\begin{corollary}\label{co15.08.1} Let $X$ be a discrete torsion Abelian group 
with cyclic $p$-components containing no elements of order $2$.
Assume that all conditions of 
Theorem $\ref{tn1}$ are fulfilled. Then the following statements are true:
\renewcommand{\labelenumi}{\rm(\roman{enumi})}
\begin{enumerate}
\item
if $\mathrm{Ker}(I+\alpha)=\{0\}$, then there is a finite 
subgroup $G$ of the group $X$ such that 
  $\mu_j$ are shifts of the Haar distribution $m_G$;
\item
if the characteristic functions of the distributions $\mu_j$ do not vanish, then
  $\mu_j$ are shifts 
of a distribution
supported in $\mathrm{Ker}(I+\alpha)$.
\end{enumerate}
\end{corollary}
\begin{proof} 
   By Theorem $\ref{tn1}$, there is a subgroup $G$ of the group $X$ such 
   that $(I+\alpha)(G)$ is a finite
group and the distributions $\mu_j$ are shifts of 
a distribution $\lambda$ supported in $G$. 

Let us prove statement (i). Inasmuch as $\mathrm{Ker}(I+\alpha)=\{0\}$ and $X$ 
is a discrete torsion Abelian group 
with cyclic $p$-components, 
we have $I+\alpha\in\mathrm{Aut}(X)$. This implies that $(I+\alpha)(G)=G$. 
Hence $G$ is a finite group and $m_{(I+\alpha)(G)}=m_{G}$. By Theorem  $\ref{tn1}$, 
the Haar distribution  $m_{(I+\alpha)(G)}$ is a factor of
$\lambda$. It follows from this that the Haar distribution  $m_G$ is a factor of
$\lambda$. Taking into account that the distribution $\lambda$ is supported in $G$, 
we have $\lambda=m_G$. 

Let us prove statement (ii). Denote by by $H$ the character group of the group $G$.
 By Theorem $\ref{tn1}$, the Haar distribution 
 $m_{(I+\alpha)(G)}$ is a factor of  $\lambda$. This implies that 
 the characteristic  function $\hat\lambda(h)$ is of the form
 \begin{equation}
 \label{15.08.1}\hat\lambda(h)=
 \begin{cases}
 \psi(h) & \text{\ if\ }\ \ h\in\mathrm{Ker}(I+\widetilde\alpha_G),
 \\ 0 & \text{\ if\ }\ \ h\notin\mathrm{Ker}(I+\widetilde\alpha_G),
 \end{cases}
 \end{equation}
where $\psi(h)$ is 
 a characteristic function on the subgroup $\mathrm{Ker}(I+\widetilde\alpha_G)$.
 Since the characteristic functions $\hat\mu_j(y)$ do not vanish,   
the characteristic function  $\hat\lambda(h)$ also does not vanish.
Taking into account (\ref{15.08.1}), this means that 
$\mathrm{Ker}(I+\widetilde\alpha_G)=H$, 
 i.e., $\widetilde\alpha_G h=-h$ for all $h\in H$. Hence $\alpha_G g=-g$ for 
all $g\in G$.
This implies that
 $G=\mathrm{Ker}(I+\alpha_G)\subset \mathrm{Ker}(I+\alpha)$, i.e., $\lambda$
is supported in $\mathrm{Ker}(I+\alpha)$.
\end{proof}

\begin{remark}
Let us discuss the question of the uniqueness of the subgroup $G$ in Theorem \ref{tn1}. 

Assume
that $\mathrm{Ker}(I+\alpha)=\{0\}$ and will prove that in this case $G$ is 
uniquely determined. Indeed, let all the conditions of Theorem \ref{tn1} be 
satisfied. It follows from item (i) of Corollary \ref{co15.08.1} that then 
$\mu_j=E_{x_j}*m_G$, $j=1, 2$, where $x_j\in X$ and $G$ is a finite subgroup of $X$. 
Suppose that $\mu_j=E_{\tilde x_j}*m_{\widetilde G}$, $j=1, 2$, where $\tilde x_j\in X$ and 
$\widetilde G$ is a finite subgroup of $X$. Hence $E_{x_1}*m_G=E_{\tilde x_1}*m_{\widetilde G}$. This implies that $G=\widetilde G$.

Assume
that $\mathrm{Ker}(I+\alpha)\not=\{0\}$. Take $x_1, x_2\in X$ such that
\begin{equation}\label{10.1}
x_1+\alpha x_2=0
\end{equation} 
and 
\begin{equation}\label{10.2}
x_0\in \mathrm{Ker}(I+\alpha).
\end{equation}
 Put $\mu_j=E_{x_j+x_0}$, $j=1, 2$.
 It follows from (\ref{10.1}) and (\ref{10.2}) that the characteristic functions
$\hat\mu_j(y)$ satisfy equation  (\ref{11.04.1}). Let $\xi_j$ be independent random
variables with values in the group $X$ and distributions $\mu_j$. Let $G$ be the subgroup
of $X$ generated by $x_0$. Then $\mu_j=E_{x_j}*E_{x_0}$, $j=1, 2$, and
statements (i)--(v) of Theorem \ref{tn1} are fulfilled for
$\lambda=E_{x_0}$ and $G$. Since $\mathrm{Ker}(I+\alpha)\not=\{0\}$, we can take 
$x_0\ne 0$. Hence $G\ne\{0\}$.   Obviously, statements (i)--(v) of Theorem \ref{tn1} 
are also fulfilled for
$\tilde\lambda=E_{0}$ and $\widetilde G=\{0\}$. 
\end{remark}

Statements (iii) and (iv) of Theorem \ref{tn1} can be strengthened 
if we assume that each of $p$-component  of the group $X$ is isomorphic to 
$\mathbb{Z}(p)$. Unlike the proof of statements (iii) and (iv) of 
Theorem \ref{tn1}, we 
 prove the corresponding statements without using 
Lemmas \ref{newle18.01.2} and \ref{s5.6}. Our proof is based on 
 the following  analogue of Heyde's theorem for discrete Abelian groups.
\begin{lemma}[{\!\!\protect\cite{FeTVP}, see also 
		\cite[Theorem 10.8]{Febooknew}}] 
 \label{ln1}
  Let $X$ be a countable discrete Abelian group containing no elements of order
 $2$. Let $\alpha$ be an   automorphism of
$X$ satisfying the condition
\begin{equation*}\label{2eq1}
\mathrm{Ker}(I+\alpha)=\{0\}.
\end{equation*}
Let $\xi_1$ and $\xi_2$ be independent random variables with values
 in the group $X$ and distributions  $\mu_1$ and $\mu_2$.
If  the conditional distribution of the linear form $L_2 =
 \xi_1 + \alpha\xi_2$ given $L_1 =  \xi_1 +
 \xi_2$ is symmetric, then    $\mu_j=m_K*E_{x_j}$, where $K$ is a finite subgroup of
  $X$  and $x_j\in X$, $j=1, 2$. Moreover,
$\alpha(K)=K$. 
\end{lemma}  
\begin{theorem}\label{ntn1} Let $X$ be a discrete Abelian group of the form  
\begin{equation*}\label{en1}
X=\mathop{\mathbf{P}^*}\limits_{p_j \in {\mathcal P}}\mathbb{Z}(p_j),
\end{equation*}
where ${\mathcal P}$ is a set of pairwise distinct prime numbers such that 
$2\notin {\mathcal P}$.
 Let  $\alpha$ be an 
 automorphism of the group $X$. 
 Let $\xi_1$ and $\xi_2$ be
independent random variables with values in $X$ and distributions
$\mu_1$ and $\mu_2$. Assume that the conditional  distribution of the 
linear form $L_2 = \xi_1 + \alpha\xi_2$ given $L_1 = \xi_1 + \xi_2$ is 
symmetric. Then  there is a   subgroup $G$ of the group $X$ and  
a distribution  
$\lambda$ supported in $G$ such that 
the following statements are true:
\renewcommand{\labelenumi}{\rm(\roman{enumi})}
\begin{enumerate}
\item
$\mu_j$ are shifts of $\lambda$;
	\item
$G$ is the minimal subgroup containing the support of $\lambda$;
\item
$G=G_1\times G_2$, where $G_1=(I+\alpha)(G)$  is a finite group and $G_2$ is a subgroup  of $G$;
\item
$\lambda=m_{G_1}*\omega$, where the distribution $\omega$ is supported in the subgroup $G_2$;
\item
if $\eta_j$ are independent identically distributed random variables 
with values in    
the group $X$  and distribution  $\lambda$, then    the conditional  
distribution of the linear form
$M_2=\eta_1 + \alpha\eta_2$ given $M_1=\eta_1 + \eta_2$ is symmetric.
\end{enumerate} 
\end{theorem}
\begin{proof}
Denote by $Y$ the character group of the group $X$. The group $Y$ is topologically 
isomorphic to the direct product of the groups 
$\mathbb{Z}(p_j)$, where ${p_j \in {\mathcal P}}$. 
To avoid introducing 
additional notation, we assume that
\begin{equation*} 
Y=\mathop{\mathbf{P}}\limits_{p_j \in {\mathcal P}}\mathbb{Z}(p_j).
\end{equation*}
Arguing as in the proof of Theorem \ref{tn1}, we can suppose that 
condition (\ref{ne20.35}) is fulfilled. We will prove that in this case 
there are distributions $\mu$ and $\omega$ such that 
the following statements are true:
\renewcommand{\labelenumi}{\rm(\Roman{enumi})}
\begin{enumerate}
\item
$\mu_1=\mu_2=\mu$;
	\item
$X$ is the minimal subgroup containing the support of $\mu$;
\item
			$X=G_1\times G_2$, where $G_1=(I+\alpha)(X)$ is a finite 
			group and  $G_2$ is a subgroup  of $X$;
			
			\item
$\mu=m_{G_1}*\omega$, where the distribution $\omega$ is supported in the subgroup $G_2$.
\end{enumerate}
Then the theorem will be proved.

Statements (I) and (II) are proved in the same way as statements (I) and (II) 
 of Theorem \ref{tn1}. Note that in the proof of statements (I) and (II) 
 we do not use Lemmas \ref{newle18.01.2} and \ref{s5.6}. 
 
 Let us prove statement (III). Each automorphism of the group 
 $\mathbb{Z}(p_j)$ is the multiplication 
by a  natural number $m$,  i.e., coincides with an endomorphism $f_{m}$, 
where $m\in \{1, 2, \dots, p_j-1\}$.  Moreover, if $m\ne p_j-1$, then 
$\mathrm{Ker}(I+f_m)=\{0\}$ and $f_{p_j-1}=-I$.  Set
$$
{\mathcal P}_1 =\{p_j \in{\mathcal P}:\alpha_{\mathbb{Z}(p_j)}\ne -I\}, \quad
{\mathcal P}_2 =\{p_j \in{\mathcal P}:\alpha_{\mathbb{Z}(p_j)}=-I\}$$$$
G_1=\mathop{\mathbf{P}^*}\limits_{p_j \in {\mathcal P}_1}\mathbb{Z}(p_j), \quad G_2=\mathop{\mathbf{P}^*}\limits_{p_j \in {\mathcal P}_2}\mathbb{Z}(p_j).
$$
Then ${\mathcal P}={\mathcal P}_1\cup{\mathcal P}_2$,  
${\mathcal P}_1\cap{\mathcal P}_2=\emptyset$ and hence 
\begin{equation}\label{13.08.2}
X=G_1\times G_2.
\end{equation}
As is easily seen, 
\begin{equation}\label{en3}
I+\alpha_{G_1}\in\mathrm{Aut}(G_1), 
\end{equation}
\begin{equation}\label{pen2}
\alpha_{G_2}=-I,
\end{equation}

Denote by  $\pi$ the  distributions on the group $G_1$ with 
the characteristic function 
\begin{equation}\label{en6}
\hat\pi(y)=\hat\mu(y), \quad y\in G_1^*.
\end{equation}
The characteristic function $\hat\mu(y)$   satisfies equation (\ref{n11.04.1}).
Since $\widetilde\alpha(G_1^*)=G_1^*$, 
the characteristic function  
 $\hat\pi(y)$ satisfies the equation
\begin{equation}\label{14.08.4}
\hat\pi(u+v )\hat\pi(u+\widetilde\alpha_{G_1} v )=
\hat\pi(u-v )\hat\pi(u-\widetilde\alpha_{G_1} v), \quad u, v \in G_1^*.
\end{equation} 
  Let $\zeta_1$ and $\zeta_2$ 
be independent identically distributed random 
variables with values in the group
$G_1$ and distribution  $\pi$. 
By  Lemma \ref{lem1}, it follows from (\ref{14.08.4}) that 
the conditional  distribution of 
the linear form $N_2 = \zeta_1 + \alpha_{G_1}\zeta_2$ given 
$N_1 = \zeta_1 + \zeta_2$ 
is symmetric.
Taking into account (\ref{en3}), we can apply Lemma  
\ref{ln1} to the discrete group $G_1$, the automorphism $\alpha_{G_1}$,   
the random variables $\zeta_1$ and $\zeta_2$, and the distributions  
$\pi_1=\pi_2=\pi$. 
We get that there is a finite subgroup $F$ of the group $G_1$ and 
 elements  $g_j\in G_1$ such that
\begin{equation}\label{en7}
\pi_j=m_F*E_{g_j},\quad j=1, 2.
\end{equation} 
Taking into account (\ref{en6}), we obtain from (\ref{en7}) that
\begin{equation*}\label{2en7}
\hat\mu(y)=\hat\pi(y)=\widehat m_F(y)(g_j, y), \quad y\in G_1^*, \ \ j=1, 2.
\end{equation*}
In view of (\ref{fe22.1}), this implies that
\begin{equation}\label{13.08.1}
|\hat\mu(y)|=1, \quad y\in A(G_1^*, F).
\end{equation}
Since condition (\ref{27ne20.35}) is fulfilled, we get from (\ref{13.08.1})
that $A(G_1^*, F)=\{0\}$, 
i.e., $G_1=F$. Hence  $G_1$ is a finite group. 
It follows from (\ref{13.08.2})--(\ref{pen2}) that $G_1=(I+\alpha)(X)$. 
Thus, statement (III) is proved.

Let us prove statement (IV). 
In view of (\ref{13.08.2}), the group $Y$ is topologically isomorphic to 
the direct product of the groups $G_1^*$ and $G_2^*$. To avoid introducing 
new notation, we suppose that $Y=G_1^*\times G_2^*$ and
denote by $(a, b)$, where $a\in G_1^*$, $b\in G_2^*$, elements of the group $Y$.
We have 
$$
\mathrm{Ker}(I+\widetilde\alpha)=A(Y, (I+\alpha)(X))=A(Y, G_1)=G_2^*.$$ 
 The 
characteristic function of the distribution $\mu$ is of the form
(\ref{en10}). Hence
\begin{equation}
 \label{13.08.3}\hat\mu(a, b)=
 \begin{cases}
 \psi(b) & \text{\ if\ }\ \ a=0,
 \\ 0 & \text{\ if\ }\ \ a\ne 0.
 \end{cases}
 \end{equation}
Let $\omega$ be the distribution on the group $X$ with the 
characteristic function
\begin{equation}
 \label{d1}\hat\omega(a, b)=\psi(b), \quad (a, b)\in Y.
  \end{equation}
Since $\hat\omega(a, b)=1$ for all $a\in G_1^*$ and $A(X, G_1^*)=G_2$, 
it follows from Lemma \ref{lem2} that
$\omega$ is supported in $G_2$. It remains to be verified that
$\mu=m_{G_1}*\omega$.
Taking into account (\ref{fe22.1}) and the fact that $A(Y, G_1)=G_2^*$, the 
characteristic 
function $\widehat m_{G_1}(a, b)$
is of the form
\begin{equation}\label{nd2}
\widehat m_{G_1}(a, b)=
\begin{cases}
1 & \text{\ if\ }\ \ a=0,
\\ 0 & \text{\ if\ }\ \ a\ne 0.
\end{cases}
\end{equation}
It follows from
(\ref{13.08.3})--(\ref{nd2}) that $\hat\mu(a, b)=
\widehat m_{G_1}(a, b)\hat\omega(a, b)$ for all $(a, b)\in Y$.
Hence $\mu=m_{G_1}*\omega$ and statement (IV) is also proved.
The theorem is completely proved.
\end{proof}
 
\section{ Heyde theorem for the direct product of 
the group of real numbers and a discrete torsion Abelian group 
with cyclic $p$-components}

Consider a group of the form $X=\mathbb{R}\times K$, where $K$ 
 is a discrete torsion Abelian group. Denote by $(t, k)$, where 
 $t\in \mathbb{R}$, $k\in K$,  elements
 of the group $X$. Let $\alpha\in\mathrm{Aut}(X)$. Since $\mathbb{R}$ 
is the connected component of the zero of the group $X$ and $K$ 
is the subgroup of $X$ consisting of elements of finite order, 
$\mathbb{R}$ and $K$ are characteristic subgroups of $X$. This implies that
 $\alpha$ acts on  elements of  the group $X$ 
 as follows: $\alpha(t, k)=(at, \alpha_Kk)$, where $a\ne 0$. We will write
 $\alpha$  in the form $\alpha=(a, \alpha_K)$.
  
We   prove in this section the following generalization of Heyde's theorem for the group 
$\mathbb{R}\times K$, where $K$  is a discrete torsion Abelian group 
with cyclic $p$-components containing no elements of order $2$. 

\begin{theorem}\label{nnth1}  Let $X=\mathbb{R}\times K$, where $K$ 
is a discrete torsion Abelian group with cyclic $p$-components  
  containing no elements of order $2$. Let  $\alpha=(a, \alpha_K)$ be a topological
automorphism of the group $X$. Let $\xi_1$ and $\xi_2$ be	independent random 
variables with values in   $X$ and distributions $\mu_1$ and $\mu_2$. Assume 
that the conditional  distribution of the 	linear form $L_2 = \xi_1 + \alpha\xi_2$ 
given $L_1 = \xi_1 + \xi_2$ is symmetric.
\renewcommand{\labelenumi}{\upshape\arabic{enumi}.}	
\begin{enumerate}
 
\item		
		If $a\ne -1$, then  there is a   subgroup $G$ of the group $K$,
		Gaussian distributions $\gamma_j$ on $\mathbb{R}$, and  
		a distribution $\omega$ supported in $G$ such that 
		the following statements are true:
		\renewcommand{\labelenumii}{\upshape(\roman{enumii})}
		\begin{enumerate}
			\item
			$\mu_j$ are shifts of the distributions  $\lambda_j=\gamma_j*\omega$; 
			\item
			$G$ is the minimal subgroup containing the support of $\omega$;
			\item
			$(I+\alpha)(G)$ is a finite
group;
\item
			the Haar distribution $m_{(I+\alpha)(G)}$ is a factor of  $\omega$;
\item			
			if $\eta_j$ are independent random variables 
with values in    
the group $X$  and distributions  $\lambda_j$, then    the conditional  
distribution of the linear form
$M_2=\eta_1 + \alpha\eta_2$ given $M_1=\eta_1 + \eta_2$ is symmetric.
		\end{enumerate} 
 
\item		
		If $a=-1$, then  there is a   subgroup $G$ of the group $K$ 
	and  a distribution $\mu$ supported
		in $\mathbb{R}\times G$ such that the following statements are true:
		\renewcommand{\labelenumii}{\upshape(\roman{enumii})}
		\begin{enumerate}
		\item
		$\mu_j$ are shifts of $\mu$; 
				\item
				$(I+\alpha)(\mathbb{R}\times G)$ is a finite
group;
\item
		the Haar distribution $m_{(I+\alpha)(\mathbb{R}\times G)}$ is a factor of  $\mu$;
		\item
		if $\eta_j$ are independent identically distributed random variables 
with values in    
the group $X$  and distributions  $\mu$, then    the conditional  
distribution of the linear form
$M_2=\eta_1 + \alpha\eta_2$ given $M_1=\eta_1 + \eta_2$ is symmetric.
		\end{enumerate} 
	\end{enumerate}
	\end{theorem}

For the proof we need the following lemma.

\begin{lemma}{\rm(\!\!\cite{F_Heyde_new})}
\label{newle1} 
Let $Y$ be an Abelian group, and let $\beta$ be 
an automorphism of $Y$ such that $I-\beta\in\mathrm{Aut}(Y)$.  
Let $f(y)$ and $g(y)$ be functions on the group $Y$ satisfying the equation
$$
f(u+v)g(u+\beta v)=f(u-v)g(u-\beta v), \quad u, v\in Y.
$$
Then $f(y)$ and $g(y)$ satisfy the equations
\begin{equation}\label{11.04.16}
f(y)=f(-(I+\beta)(I-\beta)^{-1} y)
g(-2\beta(I-\beta)^{-1} y), 
\quad y\in Y, 
\end{equation}
\begin{equation}\label{11.04.8}
g(y)=g((I+\beta)(I-\beta)^{-1} y)f(2(I-\beta)^{-1} y), 
\quad y\in Y. 
\end{equation}
Assume that the inequalities $0\le f(y)\le 1$, $0\le g(y)\le 1$, $y\in Y$, are valid.
Put $\kappa=-f_4\beta(I-\beta)^{-2}$. Let 
$\kappa^my_0=y_0$ for some $y_0\in Y$ and some natural $m$.
Then  
 \begin{equation}\label{11.04.14}
f(y_0)=g(-2\beta(I-\beta)^{-1} y_0), 
\end{equation} 
\begin{equation}\label{11.04.15}
g(y_0)=f(2(I-\beta)^{-1} y_0). 
\end{equation} 
\end{lemma}  
\noindent\textbf{\emph{Proof of Theorem \ref{nnth1}}} 
Denote by $Y$ the character group of the group $X$ and by   $L$ 
the character group of the group $K$. The group $Y$ is topologically 
isomorphic to the group $\mathbb{R}\times L$. 
Denote by $(s, l)$, 
where $s\in\mathbb{R}$, $l\in L$, elements of the group $Y$. 
By Lemma \ref{lem1}, 
the characteristic functions $\hat\mu_j(s, l)$   satisfy equation (\ref{11.04.1})
which takes the form
\begin{multline}\label{nn11.04.1}
\hat\mu_1(s_1+s_2, l_1+l_2)\hat\mu_2(s_1+as_2, 
l_1+\widetilde\alpha_Kl_2)\\=
\hat\mu_1(s_1-s_2, l_1-l_2)\hat\mu_2(s_1-as_2, 
l_1-\widetilde\alpha_Kl_2), \quad s_j\in \mathbb{R}, \ l_j \in L.
\end{multline} 

1. Assume that $a\ne -1$. Substituting $l_1=l_2=0$ in equation  (\ref{nn11.04.1}) and applying 
Lemma \ref{lem1} and Heyde's theorem  to the group $\mathbb{R}$,
 we obtain
\begin{equation*}
\hat\mu_j(s, 0)=\exp\{-\sigma_js^2+ib_js\}, 
\quad s_j\in \mathbb{R},  \ j=1, 2,
\end{equation*} 
where $\sigma_j\ge 0$, $b_j\in\mathbb{R}$. Inasmuch as  $\sigma_1+a\sigma_2=0$,
this implies that either $\sigma_1=\sigma_2=0$  or $\sigma_1>0$ and $\sigma_2>0$.

Denote by $\gamma_j$ the Gaussian
distribution on the group $\mathbb{R}$ with the characteristic function 
\begin{equation}\label{10mm11.04.1}
\hat\gamma_j(s)=\hat\mu_j(s, 0)=\exp\{-\sigma_js^2+ib_js\}, 
\quad s_j\in \mathbb{R},  \ j=1, 2.
\end{equation} 

1a. Assume that $\sigma_1=\sigma_2=0$, i.e.,
\begin{equation*}
\hat\gamma_j(s)=\hat\mu_j(s, 0)=\exp\{ib_js\}, 
\quad s_j\in \mathbb{R},  \ j=1, 2,
\end{equation*}
where $b_1+ab_2=0$. Set
$$
\lambda_j=\mu_j*E_{-b_j}, \quad j=1, 2.
$$
Then $\hat\lambda_j(s, 0)=1$ for all $s_j\in \mathbb{R}$,  $j=1, 2$. By Lemma
\ref{lem2}, the distributions $\lambda_j$ are supported in $A(X, \mathbb{R})=K$.
Moreover, the characteristic functions  $\lambda_j(0, l)$ satisfy equation (\ref{11.04.1})
on the group $L$. Let $\zeta_j$ 
be independent random variables with values in the group  $K$  and distributions 
$\lambda_j$. By Lemma  \ref{lem1}, the conditional  distribution of the linear form
$N_2=\zeta_1 + \alpha_K\zeta_2$ given $N_1=\zeta_1 + \zeta_2$ is symmetric.
The group $K$ is a discrete torsion Abelian group with cyclic $p$-components  
  containing no elements of order $2$.
The statements of the theorem  follows from Theorem \ref{tn1} 
applying to the group $K$, 
the automorphism $\alpha_K$, the random variables $\zeta_j$, 
and the distributions $\lambda_j$.  The Gaussian distributions 
$\gamma_j$ in this case are degenerated.

1b. Assume that $\sigma_1>0$ and $\sigma_2>0$.  
Putting $s_1=s_2=0$ in equation (\ref{nn11.04.1}), we get
\begin{equation}\label{nnn11.04.1}
\hat\mu_1(0, l_1+l_2)\hat\mu_2(0, l_1+\widetilde\alpha_Kl_2)=
\hat\mu_1(0, l_1-l_2)\hat\mu_2(0, l_1-\widetilde\alpha_Kl_2), \quad 
 l_j \in L.
\end{equation}
Denote by  $\omega_j$ the  distributions on the group $K$ with the characteristic functions
\begin{equation*}\label{2en6}
\hat\omega_j(l)=\hat\mu_j(0, l), \quad l\in L, \ j=1, 2.
\end{equation*}
Let $\zeta_1$ and $\zeta_2$ be independent random variables with values in the group
$K$ and distributions $\omega_1$ and $\omega_2$.
Put
$$
H=\{(0, l)\in Y:|\hat\mu_1(0, l)|=|\hat\mu_2(0,l)=1\}.
$$
By Lemma \ref{lem2}, $H$ is a closed subgroup of $Y$. 
 In addition $H^{(2)}=H$. Taking into account Lemma \ref{lem1} and
applying Lemma \ref{lem11} to the group $K$, the automorphism $\alpha_K$, 
the independent random variables   $\zeta_j$, and  the  distributions $\omega_j$, 
we reduce the 
 proof of the theorem to the case when $H=\{0\}$. We will prove that in this case
$G=K$, statements 1(ii)--1(iv) are true, and $\mu_j=\gamma_j*\omega$, $j=1, 2$, i.e., 
statement  1(i) is also true. 
Thus, the theorem will be proved.
 
It follows from $H=\{0\}$ that
\begin{equation}\label{nne20.35}
\{l\in L:|\hat\omega_1(l)|=|\hat\omega_2(l)|=1\}=\{0\}.
\end{equation}
By Lemma \ref{lem1}, the conditional  distribution of 
the linear form $N_2 = \zeta_1 + \alpha_K\zeta_2$ given 
$N_1 = \zeta_1 + \zeta_2$ 
is symmetric. Since (\ref{nne20.35}) holds, we can apply 
Corollary \ref{nnco1} to the group $K$, 
the automorphism $\alpha_K$, the random 
variables $\zeta_j$, and the
distributions $\omega_j$. We obtain that there is a 
distribution $\omega$ such that 
$\omega_1=\omega_2=\omega$,  $K$ 
is the minimal subgroup containing the support of $\omega$,
$(I+\alpha)(K)$ is a finite
group, and  the Haar distribution $m_{(I+\alpha)(K)}$ 
is a factor of of $\omega$. Thus, (\ref{nne20.35}) implies that 
 statements 1(ii)--1(iv) are true.

It remains to prove  statement  1(i), i.e., $\mu_j=\gamma_j*\omega$, $j=1, 2$. 
Taking into account that $(I+\alpha)(K)=(I+\alpha_K)(K)$,
we obtain from the above that
\begin{equation}\label{6mm11.04.1}
\hat\omega(l)=\hat\mu_1(0, l)=\hat\mu_2(0, l)=\begin{cases}\psi(l)  
&\text{\ if\ }\ \ l\in \mathrm{Ker}(I+\widetilde\alpha_K), \\ 0& 
\text{\ if\ }\ \ l\notin \mathrm{Ker}(I+\widetilde\alpha_K),
\\
\end{cases}  
\end{equation}
where $\psi(l)$ is a characteristic function 
 on 
the subgroup $\mathrm{Ker}(I+\widetilde\alpha_K)$. 

Consider equation (\ref{nn11.04.1}) supposing that 
$l_1, l_2\in \mathrm{Ker}(I+\widetilde\alpha_K)$. Inasmuch as 
$\widetilde\alpha_Kl=-l$ for 
all $l\in\mathrm{Ker}(I+\widetilde\alpha_K)$,   we have
\begin{multline}\label{1m11.04.1}
\hat\mu_1(s_1+s_2, l_1+l_2)\hat\mu_2(s_1+as_2, l_1-l_2)\\=
\hat\mu_1(s_1-s_2, l_1-l_2)\hat\mu_2(s_1-as_2, l_1+l_2), 
\quad s_j\in \mathbb{R}, \ l_j \in \mathrm{Ker}(I+\widetilde\alpha_K).
\end{multline} 
Substituting $s_1=-as$, $s_2=s$, $l_1=-l$, $l_2=l$ 
in equation  (\ref{1m11.04.1}), we get
\begin{equation}\label{2m11.04.1}
\hat\mu_1((1-a)s, 0)\hat\mu_2(0, -2l)=
\hat\mu_1(-(1+a)s, -2l)\hat\mu_2(-2as, 0), 
\quad s\in \mathbb{R}, \ l \in \mathrm{Ker}(I+\widetilde\alpha_K).
\end{equation} 
It follows from (\ref{10mm11.04.1}) that $\hat\mu_2(-2as, 0)\ne 0$ for all  
$s\in \mathbb{R}$. Taking into account that $a\ne -1$ and 
$f_2$ is an automorphism of any subgroup of
the group $L$, we obtain from (\ref{2m11.04.1}) that
there are some functions $F_1(s)$, $s\in \mathbb{R}$, and $G_1(l)$,  
$l\in \mathrm{Ker}(I+\widetilde\alpha_K)$, such that
$$
\hat\mu_1(s, l)=F_1(s)G_1(l), 
\quad s\in \mathbb{R}, \ l\in \mathrm{Ker}(I+\widetilde\alpha_K).
$$ 
Moreover, in view of $F_1(0)=G_1(0)=1$, this implies that 
$$F_1(s)=\hat\mu_1(s, 0),  \ s\in\mathbb{R}, \ \quad 
G_1(l)=\hat\mu_1(0, l), \  l\in \mathrm{Ker}(I+\widetilde\alpha_K).$$  
Hence
\begin{equation}\label{3m11.04.1}
\hat\mu_1(s, l)=\hat\mu_1(s, 0)\hat\mu_1(0, l), 
\quad s\in \mathbb{R}, \ l\in \mathrm{Ker}(I+\widetilde\alpha_K).
\end{equation} 

Substituting $s_1=s$, $s_2=-s$, $l_1= l_2=l$ in equation (\ref{1m11.04.1}) 
and arguing similarly 
we receive that 
\begin{equation}\label{4m11.04.1}
\hat\mu_2(s, l)=\hat\mu_2(s, 0)\hat\mu_2(0, l), 
\quad s\in \mathbb{R}, \ l\in \mathrm{Ker}(I+\widetilde\alpha_K).
\end{equation} 
In view of (\ref{10mm11.04.1}) and (\ref{6mm11.04.1}), we get from
(\ref{3m11.04.1}) and (\ref{4m11.04.1}) that
\begin{equation}\label{9m11.04.1}
\hat\mu_j(s, l)=\hat\gamma_j(s)\hat\omega(l), 
\quad s\in \mathbb{R}, \ l\in \mathrm{Ker}(I+\widetilde\alpha_K), \ j=1, 2.
\end{equation}

Substituting $l_1=l_2=l$, where $l_1, l_2\in\mathrm{Ker}(I-\widetilde\alpha_K)$ 
in equation
 (\ref{nnn11.04.1}), we get
\begin{equation*}\label{nnn11.04.1n}
\hat\mu_1(0, 2l)\hat\mu_2(0, 2l)=1, \quad 
 l \in \mathrm{Ker}(I-\widetilde\alpha_K).
\end{equation*}
Since $f_2$ is a topological automorphism of any closed subgroup of
the group $L$, we obtain from here
\begin{equation*}\label{m11.04.1n}
|\hat\omega(l)|=|\hat\mu_1(0, l)|=|\hat\mu_2(0, l|)=1, \quad 
 l \in \mathrm{Ker}(I-\widetilde\alpha_K).
\end{equation*}
Taking into account (\ref{nne20.35}) and (\ref{6mm11.04.1}), this 
implies that $\mathrm{Ker}(I-\widetilde\alpha_K)=\{0\}$.
Hence $I-\widetilde\alpha_K$ is a topological automorphism of 
the group $L$.
Moreover, $f_2$ is also a topological automorphism of the group $L$. 
Inasmuch as any closed subgroup of the group $L$ is characteristic,
it follows from this that $\widetilde\alpha_K$, $I-\widetilde\alpha_K$, 
and $f_2$ are one-to-one mappings of the subgroup 
$\mathrm{Ker}(I+\widetilde\alpha_K)$ onto itself.
This implies that $\widetilde\alpha_K$, $I-\widetilde\alpha_K$, 
and $f_2$ are one-to-one mappings of the set $L\setminus\mathrm{Ker}(I+\widetilde\alpha_K)$ 
onto itself.

Put $s_1=-as$, $s_2=s$, $l_1=\widetilde\alpha_Kl$, $l_2=l$, where
$l\notin\mathrm{Ker}(I+\widetilde\alpha_K)$, in equation (\ref{nn11.04.1}). 
We get
 \begin{multline}\label{5m11.04.1}
\hat\mu_1((1-a)s, (I+\widetilde\alpha_K)l)
\hat\mu_2(0, 2\widetilde\alpha_Kl)\\
=\hat\mu_1(-(1+a)s, -(I-\widetilde\alpha_K)l)\hat\mu_2(-2as, 0), 
\quad s\in \mathbb{R}, \ l\notin \mathrm{Ker}(I+\widetilde\alpha_K).
\end{multline} 
Inasmuch as $2\widetilde\alpha_Kl\notin\mathrm{Ker}(I+\widetilde\alpha_K)$, 
we obtain from (\ref{6mm11.04.1}) that 
$\hat\mu_2(0, 2\widetilde\alpha_Kl)=0$ and
for this reason 
the left-hand side of equation (\ref{5m11.04.1}) is equal to zero. 
Hence the right-hand side of equation (\ref{5m11.04.1}) is also 
equal to zero. It follows from (\ref{10mm11.04.1}) that 
$\hat\mu_2(-2as, 0)\ne 0$ for all  
$s\in \mathbb{R}$ and we conclude from (\ref{5m11.04.1}) that 
$\hat\mu_1(-(1+a)s, -(I-\widetilde\alpha_K)l)=0$. Since $a\ne -1$ 
and $I-\widetilde\alpha_K$ is
a one-to-one mapping of the set $L\setminus\mathrm{Ker}(I+\widetilde\alpha_K)$ 
onto itself, we get
\begin{equation}\label{7m11.04.1}
\hat\mu_1(s, l)=0, 
\quad s\in \mathbb{R}, \ l\notin \mathrm{Ker}(I+\widetilde\alpha_K).
\end{equation} 

Putting $s_1=s$, $s_2=-s$, $l_1=l_2=l$, where
$l\notin\mathrm{Ker}(I+\widetilde\alpha_K)$, 
in equation   (\ref{nn11.04.1}) and arguing similarly we get
\begin{equation}\label{8m11.04.1}
\hat\mu_2(s, l)=0, 
\quad s\in \mathbb{R}, \ l\notin \mathrm{Ker}(I+\widetilde\alpha_K).
\end{equation}
Taking into account (\ref{6mm11.04.1}) and (\ref{9m11.04.1}),  
we conclude from 
(\ref{7m11.04.1}) and (\ref{8m11.04.1})  that 
$$\hat\mu_j(s, l)=\hat\gamma_j(s)\hat\omega(l), 
\quad s\in \mathbb{R}, \ l\in L.$$ 
Hence $\mu_j=\gamma_j*\omega$, $j=1, 2$, i.e., statement 1(i) is proved. 
Thus, in the case when $a\ne -1$ the theorem is proved.

2. Assume now that $a=-1$, i.e.,
$\alpha=(-1, \alpha_K)$.  The proof of the theorem 
in this case is carried out according to the same scheme as the proof
of items 3 and 4 of Theorem 2.1 in \cite{F_Heyde_new}.

Put $s_1=s_2=0$ in equation (\ref{nn11.04.1}). 
Considering the resulting equation and applying Lemmas \ref{lem11}
and  \ref{lem1} to the group $K$, we can  suppose that
\begin{equation}\label{nnn11.04.10}
\{l\in L:|\hat\mu_1(0, l)|=|\hat\mu_2(0, l)|=1\}=\{0\}.
\end{equation}
We will prove that in this case statements 2(i)--2(iv) are true, if we put $G=K$.

Since $(I+\alpha)(X)=(I+\alpha_K)(K)$, it follows from Theorem \ref{tn1}
applying to the group $K$ that $(I+\alpha)(\mathbb{R}\times K)$
is a finite group. Hence statement  2(ii) is true.

Taking into account (\ref{nnn11.04.10}), it follows from Corollary \ref{nnco1} that  
$I-\alpha_K\in \mathrm{Aut}(K)$. 
Inasmuch as $\alpha=(-1, \alpha_K)$, 
this implies that $I-\alpha\in \mathrm{Aut}(X)$ and hence 
$I-\widetilde\alpha\in \mathrm{Aut}(Y)$. Put   $\nu_j=\mu_j*\bar\mu_j$. Then 
$\hat\nu_j(y)=|\hat\mu_j(y)|^2\ge 0$ for all   $y\in Y$, $j=1, 2$. 
Set  $f(y)=\hat\nu_1(y)$, $g(y)=\hat\nu_2(y)$, $y\in Y$.  
Since
the characteristic functions $\hat\mu_j(y)$   satisfy equation 
(\ref{11.04.1}), the characteristic functions   $\hat\nu_j(y)$ 
also satisfy equation   (\ref{11.04.1})  which takes the form
\begin{equation*} 
f(u+v)g(u+\widetilde\alpha  v)=f(u-v)g(u-\widetilde\alpha  v), \quad u, v\in Y.
\end{equation*}
Put  $\kappa=-f_4\widetilde\alpha(I-\widetilde\alpha)^{-2}$.   
Then   $\kappa\in \mathrm{Aut}(Y)$. It follows from   $\alpha=(-1, \alpha_K)$  
that
\begin{equation}\label{14.08.6}
\kappa=(1, -f_4\widetilde\alpha_K(I-\widetilde\alpha_K)^{-2}).
\end{equation}
  
We  will prove now that if  $f(y)\ne 0$, then  $y\in \mathrm{Ker}(I+\widetilde\alpha)$.
Note that 
$\mathrm{Ker}(I+\widetilde\alpha)=\mathbb{R}\times \mathrm{Ker}(I+\widetilde\alpha_K)$
and suppose that $f(y_1)\ne 0$ at  an element $y_1=(s_1, l_1)\in Y$.

First assume that   $l_1$ is an element of
finite order. It follows from (\ref{14.08.6}) and the fact 
that any closed subgroup of the group $L$ is characteristic 
that $\kappa^my_1=y_1$ for some $m$.
We see that all conditions of Lemma \ref{newle1},   where 
 $\beta=\widetilde\alpha$ and $y_0=y_1$,  are fulfilled. Hence
 (\ref{11.04.16})--(\ref{11.04.15}) hold.
Consider the subgroup $\langle  l_1 \rangle$ of the group $L$ generated by
the element $l_1$. Put  
$$
y_2=-2\widetilde\alpha(I-\widetilde\alpha)^{-1}y_1=
(s_1, -2\widetilde\alpha_K(I-\widetilde\alpha_K)^{-1} l_1)=(s_1, l_2),
$$
where $l_2=-2\widetilde\alpha_K(I-\widetilde\alpha_K)^{-1} l_1$. We have
$-f_2\widetilde\alpha_K(I-\widetilde\alpha_K)^{-1}\in\mathrm{Aut}(L)$. 
This implies that $l_2\in \langle l_1\rangle$. Moreover, the elements $l_1$ 
and $l_2$ have the same order  and hence $\langle l_1\rangle=\langle l_2\rangle$. 
Substituting $\beta=\widetilde\alpha$ and $y_0=y_1$ in (\ref{11.04.14}), we get
$g(y_2)\ne 0$. 
Put 
\begin{equation}\label{1nn11.04.15}
z_1=-(I+\widetilde\alpha)(I-\widetilde\alpha)^{-1} y_1=
(0, -(I+\widetilde\alpha_K)(I-\widetilde\alpha_K)^{-1}l_1).
\end{equation}
We find from   (\ref{11.04.16}) and (\ref{11.04.14}) that  then $f(z_1)=1$. 
By Lemma \ref{lem2}, $f(y)=1$ at each element  of the subgroup $\langle z_1\rangle$.
Put
\begin{equation}\label{1nnn11.04.15}
z_2=(I+\widetilde\alpha)(I-\widetilde\alpha)^{-1} 
y_2=(0, (I+\widetilde\alpha_K)(I-\widetilde\alpha_K)^{-1}l_2).
\end{equation}
Since $g(y_2)\ne 0$ and $\kappa^my_2=y_2$ for some $m$, 
substituting $\beta=\widetilde\alpha$ and $y_0=y_2$
in (\ref{11.04.15}), we find from  (\ref{11.04.8}) that  
$g(z_2)=1$. By Lemma \ref{lem2},
$g(y)=1$ at each element  of the subgroup $\langle z_2\rangle$.
As far as the elements $l_1$ and $l_2$ have the same order, 
it follows from (\ref{1nn11.04.15}) and  (\ref{1nnn11.04.15}) that the elements 
$z_1$ and $z_2$ also have the same order. 
Hence
$\langle z_1\rangle=\langle z_2\rangle$. 
Thus,  we proved that $f(y)=g(y)=1$ at each element of the subgroup
$\langle z_1\rangle=\langle z_2\rangle$. In view of $z_1, z_2\in L$, 
it follows from (\ref{nnn11.04.10}) that $z_1=0$. This implies that
$l_1\in \mathrm{Ker}(I+\widetilde\alpha_K)$ and hence 
$y_1\in  \mathrm{Ker}(I+\widetilde\alpha)$.

Assume now that   $l_1$ is an element of infinite order and 
$y_1\not\in  \mathrm{Ker}(I+\widetilde\alpha)$. 
Since $Y\setminus\mathrm{Ker}(I+\widetilde\alpha)$ is an open set and $f(y_1)\ne 0$,
we can choose a neighbourhood $U$ of the 
 element $y_1$ such that $U\subset Y\setminus\mathrm{Ker}(I+\widetilde\alpha)$
 and $f(y)\ne 0$ for all $y\in U$. It is obvious that there is an element
 $\tilde y=(\tilde s, \tilde l)\in U$ such that $\tilde l$ is an element of finite order.
 As proven above $\tilde y\in  \mathrm{Ker}(I+\widetilde\alpha)$.
 This implies that $y_1\in  \mathrm{Ker}(I+\widetilde\alpha)$.
  The obtained  contradiction shows that $y_1\in  \mathrm{Ker}(I+\widetilde\alpha)$.

Arguing similarly, we prove that if $g(y_1)\ne 0$ at  an element $y_1\in Y$, then
 $y_1\in \mathrm{Ker}(I+\widetilde\alpha)$.  
Thus, the characteristic functions $\hat\mu_j(y)$ are represented as follows 
\begin{equation}\label{nn11.04.16}
\hat\mu_j(y) = \begin{cases}\psi_j(y)  &\text{\ if\ }\ \ 
y\in \mathrm{Ker}(I+\widetilde\alpha), \\ 0& 
 \text{\ if\ }\ \ y\notin \mathrm{Ker}(I+\widetilde\alpha),
\\
\end{cases}
\end{equation}
where $\psi_j(y)$ are some characteristic functions on 
the group $\mathrm{Ker}(I+\widetilde\alpha)$. 
Considering the restriction of equation (\ref{11.04.1}) for the 
characteristic functions $\hat\mu_j(y)$ to the subgroup 
$\mathrm{Ker}(I+\widetilde\alpha)$ we are convinced that
 $\psi_1(y)=\psi_2(y)$ for all  $y\in \mathrm{Ker}(I+\widetilde\alpha)$. 
In view of
(\ref{nn11.04.16}), this implies that $\hat\mu_1(y) 
= \hat\mu_2(y)$ for all $y\in Y$.
Hence $\mu_1=\mu_2$. Put $\mu=\mu_1=\mu_2$. Thus, statement 2(i) is proved. 

Set $\psi(y)=\psi_1(y)=\psi_2(y)$. 
It follows from (\ref{nn11.04.16}) that the characteristic 
function $\hat\mu(y)$
is of the form (\ref{en10}). By statement  2(ii), $(I+\alpha)(X)$ is a finite group.
Taking into account the fact that 
$A(Y, (I+\alpha)(X))=\mathrm{Ker}(I+\widetilde\alpha)$ and (\ref{fe22.1}), 
the characteristic function $\widehat m_{(I+\alpha)(X)}(y)$
is of the form (\ref{d2}).  This implies that
$\hat\mu(y)=
\widehat m_{(I+\alpha)(X)}(y)\hat\mu(y)$ for all $y\in Y$.
Hence $\mu=m_{(I+\alpha)(X)}*\mu$, i.e.,
the Haar distribution $m_{(I+\alpha)(X)}$ is a factor of the 
 distribution $\mu$.  Statement 2(iii) is proved.  
 
 Statement 2(iv) follows from statement 2(i).
Thus, in the case when $a=-1$ the theorem is proved. 
Hence the theorem is  completely proved. 
$\hfill\Box$

\medskip

\noindent\textbf{ Acknowledgements}

\medskip

 I thank the reviewers for a careful reading of the manuscript and useful remarks.

\medskip

\noindent\textbf{Funding} The author  has not disclosed any funding.
 
\medskip

\noindent\textbf{Data Availability Statement} Data sharing 
not applicable to this article as no datasets were
generated or analysed during the current study.

\medskip
 
\noindent\textbf{Declarations}

\medskip

\noindent\textbf{Conflict of interest} The author states that there is 
no conflict of interest.

\medskip

\noindent B. Verkin Institute for Low Temperature Physics and Engineering\\
of the National Academy of Sciences of Ukraine\\
47, Nauky ave, Kharkiv, 61103, Ukraine

\medskip

\noindent e-mail:    gennadiy\_f@yahoo.co.uk

\end{document}